\pdfoutput=1
\documentclass{amsart}
\usepackage[utf8]{inputenc}
\usepackage[british]{babel}
\usepackage{amsmath}
\usepackage{amsfonts}
\usepackage{amsthm}
\usepackage{amssymb}
\usepackage{mathrsfs}
\usepackage{enumerate}
\usepackage{hyperref}
\usepackage[style=numeric, maxbibnames=10]{biblatex}
\usepackage{multirow}
\usepackage[shortlabels]{enumitem}
\usepackage{tikz}

\usepackage[protrusion=true,expansion=true,babel=true,final,]{microtype}

\setlist[enumerate]{label=(\alph*)}

\renewcommand{\epsilon}{\varepsilon}
\renewcommand{\phi}{\varphi}
\renewcommand{\theta}{\vartheta}
\newcommand{\ideal}{\trianglelefteq}

\theoremstyle{plain}
\newtheorem{theorem}{Theorem}[section]
\newtheorem{lemma}[theorem]{Lemma}
\newtheorem{proposition}[theorem]{Proposition}
\newtheorem{corollary}[theorem]{Corollary}

\theoremstyle{remark}
\newtheorem{remark}[theorem]{Remark}
\newtheorem{example}[theorem]{Example}

\newcommand{\N}{\mathbb{N}}
\newcommand{\Z}{\mathbb{Z}}

\newcommand{\C}{\mathbb{C}}

\DeclareMathOperator{\rk}{rk}

\DeclareMathOperator{\GL}{GL}
\DeclareMathOperator{\SL}{SL}
\DeclareMathOperator{\Sp}{Sp}

\DeclareMathOperator{\Irr}{Irr}
\DeclareMathOperator{\co}{co}
\DeclareMathOperator{\Ind}{Ind}
\newcommand{\T}{\mathsf{T}}
\renewcommand{\O}{\mathsf{O}}
\newcommand{\I}{\mathsf{I}}
\newcommand{\OT}{\mathsf{OT}}
\renewcommand{\H}{\mathsf{H}}
\renewcommand{\c}{\mathbf{c}}
\newcommand{\h}{\mathfrak{h}}

\DeclareMathOperator{\Spec}{Spec}

\title[Symplectic linear quotients admitting a symplectic resolution]{Towards the classification of symplectic linear quotient singularities admitting a symplectic resolution}

\author{Gwyn Bellamy}
\address{School of Mathematics and Statistics, University of Glasgow, University Place, Glasgow, G12 8QQ, UK}
\email{Gwyn.Bellamy@glasgow.ac.uk}

\author{Johannes Schmitt}
\address{Fachbereich Mathematik, Technische Universität Kaiserslautern, 67663 Kaiserslautern, Germany}
\email{schmitt@mathematik.uni-kl.de}

\author{Ulrich Thiel}
\address{Fachbereich Mathematik, Technische Universität Kaiserslautern, 67663 Kaiserslautern, Germany}
\email{thiel@mathematik.uni-kl.de}

\makeatletter
\def\blfootnote{\gdef\@thefnmark{}\@footnotetext}
\makeatother

\date{May 4, 2020}
\begin{document}

\blfootnote{\textsc{Gwyn Bellamy, School of Mathematics and Statistics, University of Glasgow, University Place, Glasgow, G12 8QQ, UK}, \texttt{Gwyn.Bellamy@glasgow.ac.uk}}
\blfootnote{\textsc{Johannes Schmitt, Fachbereich Mathematik, Technische Universität Kaiserslautern, 67663 Kaiserslautern, Germany}, \texttt{schmitt@mathematik.uni-kl.de}}
\blfootnote{\textsc{Ulrich Thiel, Fachbereich Mathematik, Technische Universität Kaiserslautern, 67663 Kaiserslautern, Germany}, \texttt{thiel@mathematik.uni-kl.de}}

\maketitle

\begin{abstract}
Over the past two decades, there has been much progress on the classification of symplectic linear quotient singularities $V/G$ admitting a symplectic (equivalently, crepant) resolution of singularities. The classification is almost complete but there is an infinite series of groups in dimension 4—the symplectically primitive but complex imprimitive groups—and 10 exceptional groups up to dimension 10, for which it is still open. In this paper, we treat the remaining infinite series and prove that for all but possibly $39$ cases there is no symplectic resolution. We thereby reduce the classification problem to finitely many open cases. We furthermore prove non-existence of a symplectic resolution for one exceptional group, leaving $39+9=48$ open cases in total. We do not expect any of the remaining cases to admit a symplectic resolution.
\end{abstract}

\section{Introduction to the problem and current status}
Recall that a \emph{smooth symplectic variety} is a complex algebraic variety $X$ equipped with a regular, closed, and non-degenerate $2$-form $\omega$, i.e.\ there is a ``smooth'' family of symplectic forms $\omega_x$ on the tangent spaces $T_x X$ of $X$. As the dimension of the tangent space at a singular point is greater than that at smooth points, it is not clear how to extend this concept to singular varieties—and what to gain from this.

In 2000, Beauville \cite{Beauville} proposed such an extension: a (possibly singular) \emph{symplectic variety} is a normal variety $X$ with a symplectic form $\omega$ on its smooth part $X^{\mathrm{sm}}$ such that for any resolution $\pi \colon \widetilde{X} \to X$ of singularities (i.e.\ a proper birational morphism with $\widetilde{X}$ smooth) the pullback of $\omega$ to $\pi^{-1}(X^{\mathrm{sm}})$ extends to a regular 2-form on all of $\widetilde{X}$. Since two given resolutions are dominated by a common resolution, it is enough to check this property only for one particular resolution. Singularities of a symplectic variety are called \emph{symplectic singularities}. They are rational Gorenstein \cite{Beauville}. In retrospect, this definition seems natural but Beauville was originally motivated by the analogy between rational Gorenstein singularities and Calabi--Yau manifolds. Symplectic singularities have become a very important and influential subject, not just in algebraic geometry but also in representation theory \cite{Fu-Survey,SymplecticDuality}.

If the pullback of $\omega$ to $\pi^{-1}(X^{\mathrm{sm}})$ extends not just to a regular 2-form but to a symplectic form, the resolution $\pi$ is called \emph{symplectic}. This is the kind of resolution one would like to have in this context. In light of the minimal model program \cite{BCHM}, we moreover want the resolution to be a \emph{projective} morphism. From now on, by ``resolution'' we will always mean a ``projective resolution''. The canonical class $K_X$ of a symplectic variety $X$ is trivial since it is trivialized by $\wedge^n \omega$, where $\dim X  = 2n$. Hence, if $\pi \colon \widetilde{X} \to X$ is a symplectic resolution, then $K_{\widetilde{X}}$ is trivial as well. In particular, $\pi^* K_X = K_{\widetilde{X}}$, i.e.\ $\pi$ is a \emph{crepant} resolution. Conversely, a crepant resolution of a symplectic variety is symplectic \cite{Fu-Symp-Res-Nilp}. \\

One important class of examples of symplectic varieties are the \emph{symplectic linear quotients}: quotients $V/G=\Spec \C\lbrack V \rbrack^G$ for $V$ a finite-dimensional symplectic complex vector space and $G < \mathrm{Sp}(V)$ a finite group of symplectic automorphisms of $V$. Here, the symplectic form on the smooth part of $V/G$ is induced by the symplectic form on $V$; see \cite{Beauville}. Note that we always have $\Sp(V) \leq \SL(V)$ and that there is equality for $V=\C^2$, so 2-dimensional symplectic linear quotients are precisely the \emph{Kleinian singularities}. It is known that Kleinian singularities admit a unique minimal resolution, and the minimal resolution is crepant (thus symplectic). In higher dimensions, the situation is much more difficult and interesting. \\

\noindent\textbf{Problem.} Which symplectic linear quotients $V/G$ admit a symplectic resolution? \\

There has been much progress on this problem over the past two decades—more on this below. The classification is almost complete but there is an infinite series of groups and 10 exceptional groups for which it is still open. In this paper, we treat the infinite series (and one exceptional group) and reduce the classification problem to finitely many open cases. The first major step in the classification is due to Verbitsky \cite{Verbitsky}.

\begin{theorem}[Verbitsky]
  If $V/G$ admits a symplectic resolution (not necessarily projective), then the finite group $G$ is generated by \emph{symplectic reflections}, i.e.\ by elements $s \in G$ whose fixed space is of codimension 2 in~$V$.
\end{theorem}

Groups generated by symplectic reflections are called \emph{symplectic reflection groups}. Note that despite the terminology, a symplectic reflection group is not just a group $G$ but a pair $(V,G)$. It is clear that for the main problem we need to consider pairs $(V,G)$ only up to change of basis. Symplectic reflection groups up to conjugacy have been classified by Cohen \cite{Cohen80} in 1980. In dimension 2 the symplectic linear quotients are just the Kleinian singularities and all of them admit a symplectic resolution. So, we assume from now on we are in dimension $\geq 4$. It is sufficient to only consider \emph{symplectically irreducible} pairs $(V,G)$, i.e.\ pairs for which there is no proper non-zero symplectic subspace of $V$ invariant under $G$, since any pair is a direct sum of symplectically irreducible pairs. The irreducible ones split into four classes as illustrated in Figure \ref{fig:symptypes}.
\begin{figure}[htbp]
  \begin{tikzpicture}[scale=0.8]
    \node[align=center] at (0, 0) (root) {Symplectic\\ reflection groups};
    \node[align=center] at (-3.5, -2) (improper) {Improper\\(complex\\ reflection\\ groups)} edge[bend left] (root);
    \node[align=center] at (3.5, -2) (proper) {Proper} edge[bend right] (root);

    \node[align=center] at (6.5, -4) (primitive) {Symplectically\\ primitive} edge[bend right] (proper);
    \node[align=center] at (0.5, -4) (imprimitive) {Symplectically\\ imprimitive\\\cite[Thm.\ 2.6, 2.9]{Cohen80}} edge[bend left] (proper);

    \node[align=center] at (8.5, -6) {Complex\\ primitive\\ \cite[Thm.\ 4.2]{Cohen80}} edge[bend right] (primitive);
    \node[align=center] at (4.5, -6) {\textbf{Complex}\\ \textbf{imprimitive}\\ \cite[Thm.\ 3.6]{Cohen80}} edge[bend left] (primitive);

  \end{tikzpicture}
  \caption{\label{fig:symptypes}The different classes of symplectic reflection groups in dimension \(\geq 4\) in Cohen's classification.}
\end{figure}
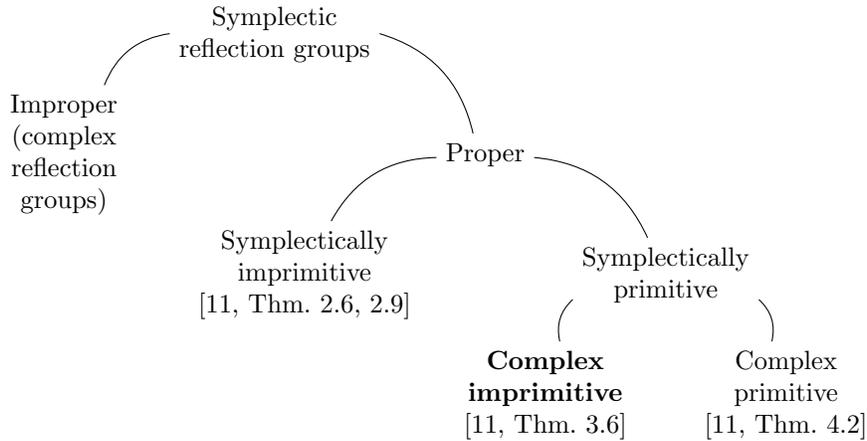

If \(G\) preserves a Lagrangian subspace \(L\subseteq V\), we say that \(G\) is \emph{improper}. In this case, \(G\leq \GL(L)\) is a complex reflection group and \(V\cong L\oplus L^\ast\) as a \(G\)-module. Complex reflection groups up to conjugacy were classified by Shephard and Todd \cite{ShephardTodd54}.
In work of Etingof--Ginzburg \cite{EtingofGinzburg02}, Gordon \cite{Gordon-Baby}, and Bellamy \cite{Bellamy09} it is proven that in this case $V/G$ admits a symplectic resolution if and only if $G$ is the group $G(m,1,n) = C_m \wr S_n$ or the exceptional group $G_4$ in the Shephard--Todd notation. All symplectic resolutions up to isomorphism were explicitly constructed for $G(m,1,n)$ by Bellamy–Craw \cite{Bellamy-Craw} and for $G_4$ by Lehn–Sorger \cite{Lehn-Sorger}.

We call a proper (i.e.\ not improper) group \(G\) \emph{symplectically imprimitive}, if there exists a non-trivial decomposition \(V = V_1\oplus \cdots \oplus V_k\) into symplectic subspaces such that, for all \(g\in G\) and all \(i\), there exists \(j\) such that \(g(V_i) = V_j\).
These groups split into infinite families given in \cite[Theorem 2.6]{Cohen80} (in dimension 4) and \cite[Theorem 2.9]{Cohen80} (in dimension greater than 4).
Linear quotients of these groups are treated in \cite{BellamySchedler16}, where the above question is answered for almost all cases.
The remaining ones are covered by \cite{Yamagishi18}.
The only groups in this class for which the linear quotient admits a symplectic resolution are the groups \(K\wr S_n\) with a finite group \(K\leq \SL_2(\C)\) and the group \(Q_8\times_{\Z/2\Z} D_8\) considered in \cite{BellamySchedler13}. All resolutions in the latter case were explicitly constructed by Donten-Bury--Wi\'{s}niewski~\cite{81resolutions}.

This leaves only the proper groups \(G\) which are \emph{symplectically primitive}, that is, groups for which no decomposition as above exists.
We still may have a decomposition into non-symplectic subspaces, so those groups may be \emph{complex primitive} or \emph{complex imprimitive}.
The complex primitive groups are given in \cite[Table III]{Cohen80}.
In \cite{BellamySchedler16}, the authors prove that for three of them (\(W(Q)\), \(W(S_3)\), \(W(T)\)) no symplectic resolution exists.
Using the same strategy, we prove in Section \ref{sec:onemoregroup}:

\begin{theorem}
The symplectic linear quotient associated to the symplectic reflection group \(W(S_2)\) does not admit a symplectic resolution.
\end{theorem}

This leaves the 9 groups coming from the root systems \(O_1\), \(O_2\), \(O_3\), \(P_1\), \(P_2\), \(P_3\), \(R\), \(S_1\), \(U\) in \cite[Table III]{Cohen80}, for which the problem is still open. \\

Apart from Section \ref{sec:onemoregroup} we consider in this paper the last class of groups, namely the symplectically primitive but complex imprimitive ones as given in \cite[Theorem 3.6]{Cohen80}. This is an infinite class of groups in dimension 4. The main result of this paper is:

\begin{theorem}
  For all but possibly finitely many of the symplectic reflection groups $(V,G)$ which are symplectically primitive but complex imprimitive the associated symplectic linear quotient $V/G$ does not admit a symplectic resolution.
\end{theorem}

The cases not covered by our theorem are explicit. By theoretical arguments, we reduce it from infinitely many to 73 open cases, see Theorem \ref{thm:nores} and Table \ref{tab:open}. Using computer calculations with the software package \textsc{Champ} \cite{Thiel15} developed by the third author, we further reduce this to 39 open cases, see Table \ref{tab:open2}.

We have thereby reduced the classification problem to finitely many (precisely, $39+9=48$) open cases. We expect none of them admits a symplectic resolution but we currently cannot prove this.\\


%
%
%

We want to mention one key tool that was used to prove non-existence of a symplectic resolution in many cases and that we will use as well: the \emph{symplectic reflection algebras} $\H_\c(V,G)$ associated to $(V,G)$ by Etingof and Ginzburg \cite{EtingofGinzburg02}. These non-commutative algebras form a flat family of deformations of the ``skew coordinate ring'' $\C \lbrack V \rbrack \rtimes G$ of $V/G$. Their centres yield a flat family of deformations of the coordinate ring $\C \lbrack V \rbrack^G$ of $V/G$. The parameter space for the $\c$ is the space $\C^{\mathcal{S}(G)/G}$, where $\mathcal{S}(G)$ is the set of symplectic reflections in $G$ and $\mathcal{S}(G)/G$ denotes the set of \(G\)-conjugacy classes of symplectic reflections. The algebra $\H_\c(V,G)$ is finite over its centre, so all its irreducible modules are finite-dimensional. In fact, their dimension is bounded above by the order of $G$, see \cite{EtingofGinzburg02}. The following theorem is a combination of theorems by Etingof–Ginzburg \cite{EtingofGinzburg02} and by Ginzburg–Kaledin \cite{Ginzburg-Kaledin}.

\begin{theorem}[Etingof--Ginzburg, Ginzburg--Kaledin]
  \label{thm:EtingofGinzburgKaledin}
 If $V/G$ admits a symplectic resolution, then there is a parameter $\c$ such that the dimension of all irreducible $\H_\c(V,G)$-modules is equal to the order of $G$.
\end{theorem}

In fact, the converse holds as well \cite{Namikawa-Poisson-Def} but most relevant for us is the negation of the above theorem: if for all $\c$ there is an irreducible $\H_\c(V,G)$-module of dimension less than the order of $G$, then $V/G$ does not admit a symplectic resolution. This is the strategy we will pursue in this paper. A key concept will be that of \emph{rigid} representations introduced by the first and third authors in \cite{BellamyThiel16}. Before we come to this, we first need to collect and prove several properties about the reflection groups in question.

\begin{remark}
As noted by the referee, choosing a suitable normal subgroup $K$ of $G$ may allow one to apply the geometric techniques of \cite{BellamySchedler16} in order to show that some of the remaining open cases of symplectically primitive but complex imprimitive groups do not give rise to quotient singularities admitting a symplectic resolution. Specifically, if there exists a symplectic resolution $Y \to V/K$ such that the action of $G/K$ on $V/K$ lifts to $Y$, then the existence of particularly bad singularities on $Y/(G/K)$ implies that $V/G$ cannot admit a symplectic resolution. However, we do not expect that the representation-theoretic techniques employed in this article (together with the computational data from \cite{Thiel15}) can take us any further in deciding if the remaining cases admit a symplectic resolution, unless new ideas are introduced.
\end{remark}

\subsection*{Acknowledgements}
We would like to thank the referee for a very detailed reading of the paper and several useful comments. Moreover, we would like to thank Gunter Malle for comments. This is a contribution to Project-ID 286237555 -- TRR 195 -- by the Deutsche Forschungsgemeinschaft (DFG, German Research Foundation).

\section{Primitive complex reflection groups}

In Cohen's classification \cite{Cohen80}, symplectically primitive reflection groups come as complexification of primitive quaternionic groups and may be complex imprimitive or primitive.
Here, we consider the first case.
These are given by four infinite families of groups all acting on \(\C^4\) by \cite[Theorem 3.6]{Cohen80}.
Each of them is constructed using an infinite family of subgroups of \(\GL_2(\C)\).
We will first describe these groups in more detail before we move on to the construction of the symplectic groups in the next section.

For any \(d\in \Z_{\geq 1}\) let
\[\mu_d := \left\langle\begin{pmatrix} \zeta_d & 0\\ 0 & \zeta_d\end{pmatrix}\right\rangle,\]
where \(\zeta_d\in\C\) is a primitive \(d\)-th root of unity.
Let \(\T\), \(\O\) and \(\I\) be the binary tetrahedral, binary octahedral and binary icosahedral group respectively, which are subgroups of \(\SL_2(\C)\).
Of course, these are only defined up to conjugacy, but there is no need to fix a representative for what follows.
See \cite[393]{Cohen76} for such an explicit description.

We have \(\T\ideal \O\) with \(\O/\T \cong C_2\), so \(\O = \langle \T, \omega \rangle\) for some \(\omega\in \O\).
We follow \cite[392]{Cohen76} to construct a further group \(\OT_d\) for any \(d\in \Z_{\geq 1}\) (or \((\mu_{2d}\mid\mu_d;\O\mid \T)\) in Cohen's notation).
For \(d\in \Z_{\geq 1}\) let \[\phi : \mu_{2d}/\mu_d \to \O/\T\] be the isomorphism defined by \(\phi\big(\overline{\zeta_{2d}I_2}\big) = \overline{\omega}\).
Set \[\mu_{2d}\times_{\phi}\O = \{(z, g)\in \mu_{2d}\times \O\mid\phi(z\mu_d) = g\T\},\] and let \(\OT_d\) denote the image of \(\mu_{2d}\times_{\phi}\O\) in \(\GL_2(\C)\) under the natural multiplication map.
That means, we have \[\OT_d = \bigcup_{\substack{k = 0\\k\text{ even}}}^{2d - 1} \zeta_{2d}^k\T \cup \bigcup_{\substack{k = 1\\k\text{ odd}}}^{2d - 1} \zeta_{2d}^k\omega\T.\]

The infinite families of subgroups of \(\GL_2(\C)\) used to constructed the symplectic reflection groups in the next section are the following:
\begin{enumerate}[(1)]
  \item\label{groups1} \(\mu_d\T\), with \(d\) a multiple of 6,
  \item\label{groups2} \(\mu_d\O\), with \(d\) a multiple of 4,
  \item\label{groups3} \(\mu_d\I\), with \(d\) a multiple of 4, 6, or 10,
  \item\label{groups4} \(\OT_{2d}\), with \(d\) an odd multiple of 1 or 2 (i.e.\ \(d\) not divisible by 4).
\end{enumerate}

\begin{lemma}
  \label{lem:centres}
  We have \(Z(\mu_d\T) = Z(\mu_d\O) = Z(\mu_d\I) = Z(\OT_d) = \mu_d\), for all even \(d\in\Z_{\geq 1}\).
\end{lemma}

\begin{proof}
  We have \(\{\pm I_2\}\subseteq \mu_d\) for even \(d\) and \(Z(\mu_d \T)\cap \T \subseteq Z(\T)=\{\pm I_2\}\) (and analogously for \(\O\) and \(\I\)), which settles the first three groups.

  Let now \(g\in Z(\OT_d)\).
  Note that \(\OT_d \subseteq \mu_{2d}\O\), so for any \(h\in \OT_d\), there exist a \(z\in\mu_{2d}\) and an \(h'\in \O\), such that \(h = z h'\).
  Then \(gh = hg\) implies \(gz h' = zh' g\), so \(gh' = h'g\).
  It follows \(g\in Z(\mu_{2d}\O) = \mu_{2d}\), so \(Z(\OT_d)\leq \mu_{2d}\).
  Since \(\mu_{2d}\cap \OT_d = \mu_d\) and clearly \(\mu_d \subseteq Z(\OT_d)\), it follows \(\mu_d = Z(\OT_d)\).
\end{proof}

\begin{lemma}
  \label{lem:dets}
  For any group \(G\) in \ref{groups1} to \ref{groups4} and any \(g\in G\) we have \(({\det g})I_2\in Z(G) = \mu_d\).
  More precisely, we have \(\{({\det g})I_2\mid g\in G\} = \mu_{d/2}\), if \(G\) belongs to \ref{groups1}, \ref{groups2}, or \ref{groups3} and \(\{({\det g})I_2\mid g\in G\} = \mu_d\) if \(G\) belongs to \ref{groups4}.
\end{lemma}
\begin{proof}
  Let \(G = \mu_d \T\) with \(d\) a multiple of 6.
  Then the claim follows directly since \(\T\leq \SL_2(\C)\) and \((\det g)I_2\in \mu_{d/2}\) for \(g\in \mu_d\).
  The same holds for the groups in \ref{groups2} and \ref{groups3}.

  Let \(G = \OT_d\) with \(d\) a multiple of 2 not divisible by 8.
  Then \(G\subseteq \mu_{2d}\O\), so any non-trivial determinant comes from an element \(\zeta_{2d}^kg\) with a primitive \(2d\)-th root of unity \(\zeta_{2d}\), \(g\in\O\) and \(0\leq k < 2d\).
  Then \(\det \zeta_{2d}^kg = \zeta_d^k\in Z(G)\).
  For the second claim notice that for any \(0\leq k < 2d\) either \(\zeta_{2d}^kI_2\in G\) or \(\zeta_{2d}^k\omega\in G\), so we obtain indeed all elements of \(Z(G)\) as determinants.
\end{proof}

\begin{lemma}
  \label{lem:OInosubs}
  The groups \(\O\) and \(\I\) are not conjugate to any subgroup of \(\mu_d \T\) for even \(d\in \Z_{\geq 1}\).
\end{lemma}
\begin{proof}
  Assume there is an embedding \(\O\hookrightarrow \mu_d \T\) for an even \(d\).
  Then we also would have an injective map \(\O/Z(\O)\hookrightarrow \mu_d\T/Z(\mu_d\T)\), since the preimage of \(Z(\mu_d\T)\) must be contained in \(Z(\O)\).
  But \[|{\mu_d\T/Z(\mu_d\T)}| = |{\mu_d \T/\mu_d}| = \frac{|{\T}|}{2} = 12\] and \[|{\O/Z(\O)}| = \frac{|{\O}|}{2} = 24,\] so this is not possible.
  The same reasoning holds for \(\I\) in place of \(\O\) since \(|{\I/Z(\I)}| = 60\).
\end{proof}

For groups \(G, H\leq \GL_2(\C)\), we write \(H\leq_g G\) if \(gHg^{-1}\leq G\) with \(g\in \GL_2(\C)\).

\begin{lemma}
  \label{lem:OnotsubOT}
  The group \(\O\) is not conjugate to any subgroup of \(\OT_{2d}\) for any \(d\in \Z_{\geq 1}\).
\end{lemma}
\begin{proof}
  Assume \(\O\leq_g \OT_{2d}\) for a \(g\in \GL_2(\C)\) and let \(h\in \O\).
  By the explicit description of \(\OT_{2d}\) above, we may distinguish two cases.

  First assume \(ghg^{-1}= \zeta_{4d}^k \omega t\) for some \(t\in \T\) and \(1\leq k < 4d\) odd.
  But this would imply \(\det\big(\zeta_{4d}^kI_2) = 1\), so \(k\) must be a multiple of \(2d\) in contradiction to \(k\) being odd.

  Hence we must have \(gh g^{-1}= \zeta_{4d}^kt\) for some \(t\in \T\) and \(0\leq k < 4d\) even.
  As this holds for all \(h\in \O\), it follows \(\O\leq_g\mu_{4d}\T\) in contradiction to Lemma \ref{lem:OInosubs}.
\end{proof}

\begin{lemma}
  \label{lem:OTsubOT}
  There exists \(g\in \GL_2(\C)\) with \(\OT_{2d} \leq_g \OT_{2d'}\) for \(d\) and \(d'\) both not divisible by 4 if and only if \(d\) divides \(d'\) with \(d'/d\) odd.
\end{lemma}
\begin{proof}
  Assume \(\OT_{2d}\leq_g \OT_{2d'}\) for a \(g\in \GL_2(\C)\).
  We have \(\zeta_{4d}\omega\in\OT_{2d}\) so \(g\zeta_{4d}\omega g^{-1}\in\OT_{2d'}\) and hence \[\det(\zeta_{4d}\omega)I_2 = \zeta_{4d}^2I_2\in Z(\OT_{2d'}) = \mu_{2d'},\] by Lemma \ref{lem:dets}.
  So \(\zeta_{4d}^2 = \zeta_{2d'}^k\) for some \(0\leq k < 2d'\), which already shows \(d\mid d'\).
  Now assume that \(k = d'/d\) is even.
  Then the only elements of \(\OT_{2d'}\) having determinant \(\zeta_{2d'}^k\) lie in \(\zeta_{4d'}^k\T\).
  But then we would have \(g\zeta_{4d}\omega g^{-1}\in \mu_{4d'}\T\), so \(g\omega g^{-1}\in \mu_{16dd'}\T\) in contradiction to Lemma \(\ref{lem:OInosubs}\).
\end{proof}

Every group \(G\) in \ref{groups1} to \ref{groups4} contains a primitive complex reflection group of rank 2.
These groups are the exceptional groups \(G_4\) to \(G_{22}\) in the classification by Shephard and Todd \cite{ShephardTodd54} by \cite[Theorem 3.4]{Cohen76}.
Following \cite{Cohen76}, we can identify the groups \(G_5\) and \(G_7\) to \(G_{22}\) with the groups in \ref{groups1} to \ref{groups4} for ``small" values of \(d\), see Table \ref{tab:comprefl}.

\begin{table}[htbp]
  \centering
  \begin{tabular}{c|c||c|c}
    \multirow{2}{*}{group} & Shephard-Todd & \multirow{2}{*}{group} & Shephard-Todd \\
          & number        &       & number        \\
    \hline
    \hline
    \(\mu_6 \T\)    &  5 & \(\mu_{12} \T\) &  7 \\
    \hline
    \(\mu_4 \O\)    & 13 & \(\mu_8 \O\)    &  9 \\
    \(\mu_{12} \O\) & 15 & \(\mu_{24} \O\) & 11 \\
    \hline
    \(\mu_4 \I\)    & 22 & \(\mu_6 \I\)    & 20 \\
    \(\mu_{10} \I\) & 16 & \(\mu_{12} \I\) & 21 \\
    \(\mu_{20} \I\) & 17 & \(\mu_{30} \I\) & 18 \\
    \(\mu_{60} \I\) & 19 & & \\
    \hline
    \(\OT_2\) & 12 & \(\OT_4\) & 8 \\
    \(\OT_6\) & 14 & \(\OT_{12}\) & 10 \\
  \end{tabular}
  \caption{\label{tab:comprefl}Primitive complex reflection groups}
\end{table}

We now want to describe the largest complex reflection group contained in \(G\).
Let \(G'\) be any primitive complex reflection group contained in \(G\).
Then the largest reflection group \(G_0\), i.e.\ the group generated by the reflections in \(G\), must be primitive too, since it contains \(G'\).
Hence \(G_0\) must be conjugate to one of the groups \(G_4\) to \(G_{22}\) in the classification by Shephard and Todd \cite{ShephardTodd54}.

To reduce the number of cases one has to consider in the proof of the next proposition, we computed which groups of the table are (conjugate to) a subgroup of another group using Magma \cite{Magma}.
We summarize the results in Table \ref{tab:subs}.
(Note that the groups \(G_4\) and \(G_6\) do not contain any other group.)

\begin{table}[htbp]
  \centering
  \begin{tabular}{c|l}
    group & is (conjugate to) a subgroup of \\
    \hline
    \multirow{2}{*}{\(\mu_6\T\)} & \(\mu_{12}\T\), \(\mu_d\O\) for \(d\in\{12, 24\}\), \(\mu_d\I\) for \(d\in\{6, 12, 30, 60\}\),\\
    &\(\OT_{2d}\) for \(d\in\{3, 6\}\)\\

    \(\mu_{12}\T\) & \(\mu_{12}\O\), \(\mu_{24}\O\), \(\mu_{12}\I\), \(\mu_{60}\I\), \(\OT_{12}\)\\

    \(\mu_4 \O\) & \(\mu_8\O\), \(\mu_{12}\O\), \(\mu_{24}\O\)\\

    \(\mu_8 \O\) & \(\mu_{24}\O\)\\

    \(\mu_{12} \O\) & \(\mu_{24}\O\) \\

    \(\mu_{24} \O\) &\\

    \(\mu_4 \I\) & \(\mu_{12} \I\), \(\mu_{20} \I\), \(\mu_{60} \I\)\\

    \(\mu_6 \I\) & \(\mu_{12} \I\), \(\mu_{30} \I\), \(\mu_{60} \I\)\\

    \(\mu_{10} \I\) & \(\mu_{20} \I\), \(\mu_{30} \I\), \(\mu_{60} \I\)\\

    \(\mu_{12} \I\) & \(\mu_{60} \I\)\\

    \(\mu_{20} \I\) & \(\mu_{60}\I\) \\

    \(\mu_{30} \I\) & \(\mu_{60}\I\) \\

    \(\mu_{60} \I\) & \\

    \(\OT_2\) & \(\mu_d\O\) for \(d\in\{4, 8, 12, 24\}\), \(\OT_6\)\\

    \(\OT_4\) & \(\mu_8\O\), \(\mu_{24}\O\), \(\OT_{12}\)\\

    \(\OT_6\) & \(\mu_{12}\O\), \(\mu_{24}\O\)\\

    \(\OT_{12}\) & \(\mu_{24}\O\)\\

  \end{tabular}
  \caption{\label{tab:subs}Subgroup relations}
\end{table}

\begin{samepage}
\begin{proposition}
  \label{prop:reflsub}
  For the groups \(G\) in \ref{groups1} to \ref{groups4} the largest complex reflection group \(G_0\subseteq \GL_2(\C)\) contained in \(G\) is as follows:
  \begin{enumerate}[(a)]
    \item If \(G = \mu_d\T\) then \(G_0 = \mu_{d_0}\T\) with \(d_0\in\{6, 12\}\) the largest number dividing \(d\).
    \item If \(G = \mu_d\O\) then \(G_0 = \mu_{d_0}\O\) with \(d_0\in\{4, 8, 12, 24\}\) the largest number dividing \(d\).
    \item If \(G = \mu_d\I\) then \(G_0 = \mu_{d_0}\I\) with \(d_0\in\{4, 6, 10, 12, 20, 30, 60\}\) the largest number dividing \(d\).
    \item If \(G = \OT_{2d}\) then \(G_0 = \OT_{2d_0}\) with \(d_0\in\{1, 2, 3, 6\}\) the largest number dividing \(d\), such that \(d/d_0\) is odd.
  \end{enumerate}
  In each case we have \(G_0\ideal G\) and \(G/G_0 \cong \mu_{d'}\) with \(d' := d/d_0\).
\end{proposition}
\end{samepage}

\begin{proof}

  \begin{enumerate}[(a),leftmargin=0pt,itemindent=*]
    \item Let \(G = \mu_d\T\), \(d\) a multiple of 6.
      Then clearly \(\mu_6\T\leq G\), so by the above discussion we have to consider the groups in the first row of Table \ref{tab:subs}.

      The group \(\mu_{12}\T\) is a subgroup of \(G\) if and only if \(d\) is a multiple of 12.

      For any \(g\in \GL_2(\C)\), we cannot have \(\mu_{\tilde{d}}\O\leq_g G\) or \(\mu_{\tilde{d}}\I\leq_g G\) for any \(\tilde{d}\) since this would imply \(\O\leq_g G\) or \(\I\leq_g G\) which does not hold by Lemma \ref{lem:OInosubs}.

      Assume finally \(\OT_{2\tilde{d}}\leq_g G\) for a \(g\in \GL_2(\C)\).
      Then for all \(h\in \O\) we have \(ghg^{-1} = \zeta_d^kt\) or \(g\zeta_{4\tilde{d}}h g^{-1} = \zeta_{d}^kt\) for some \(0\leq k < d\) and \(t\in \T\).
      But then \(gh g^{-1}\in \mu_{4\tilde{d}d}\T\), so \(\O\leq_g \mu_{4\tilde{d}d}\T\) in contradiction to Lemma \ref{lem:OInosubs}.

      So the largest complex reflection group in \(G\) is \(\mu_{d_0}\T\) with
      \[d_0 :=
        \begin{cases}
          6,& d\text{ is an odd multiple of 6,}\\
          12,& d\text{ is an even multiple of 6}
        \end{cases}
      \]
      and clearly \(G/G_0 \cong \mu_{d/d_0}\).
    \item Let \(G = \mu_d\O\), \(d\) a multiple of 4.
      Then \(\mu_4\O\leq G\), so \(\mu_4\O\leq G_0\) and we only have to consider the supergroups of \(\mu_4\O\) in Table \ref{tab:subs}.
      This already finishes this case.
    \item Let \(G = \mu_d\I\), \(d\) a multiple of 4, 6, or 10.
      Then \(G\) for sure contains \(\mu_4\I\), \(\mu_6\I\) or \(\mu_{10}\I\) and Table \ref{tab:subs} assures us that the only subgroups possible are of the form \(\mu_{d_0}\I\).
    \item Let \(G = \OT_{2d}\) for a \(d\) not divisible by 4.
      By Lemma \ref{lem:OTsubOT}, \(\OT_{2d_0}\) is a subgroup of \(\OT_{2d}\) if and only if \(d_0\) divides \(d\) and \(d/d_0\) is odd.
      Choosing the largest such \(d_0\in\{1, 2, 3, 6\}\) we hence obtain the largest reflection group of type \(\OT_{2d_0}\) contained in \(\OT_{2d}\).
      Such a \(d_0\) always exists since \(d\) is either an odd multiple of 1 or of 2.
      Consulting Table \ref{tab:subs} again, it remains to prove \(\mu_{\tilde{d}}\O\not\leq_g G\) for any \(\tilde{d}\in\{4, 8, 12, 24\}\) and any \(g\in \GL_2(\C)\).
      This follows directly with Lemma \ref{lem:OnotsubOT}.

      Lastly, we prove \(G/G_0 \cong \mu_{d/d_0}\).
      Set \(d' := d/d_0\) and define \(\phi: G\to\mu_{d'}\) by \(\phi(\zeta_{4d}^kg) := \zeta_{d'}^kI_2\) for all \(0 \leq k < 4d\) and \(g\in \O\), such that \(\zeta_{4d}^kg\in G\).
      Let \(\zeta_{4d}^kg\in \ker\phi\).
      Then \(d' \mid k\), so \(k = d'l\) for some \(l\in \N\), where \(l\) is odd if and only if \(k\) is odd, since \(d'\) is odd.
      Hence \(\zeta_{4d}^kg = \zeta_{4d_0}^lg\in \OT_{2d_0}\).
      As \(\phi\) is surjective it follows \(G/\OT_{2d_0}\cong \mu_{d'}\).\qedhere
  \end{enumerate}
\end{proof}

\section{Imprimitive symplectic reflection groups}

We are now ready to describe the already mentioned four families of imprimitive symplectic reflection groups which are symplectically primitive.

For a matrix \(g\in \GL_2(\C)\), set \[g^\vee := \begin{pmatrix} g & 0\\ 0 & (g^\top)^{-1}\end{pmatrix}\in \GL_4(\C).\]
For any subset (in particular group) \(G\subseteq\GL_2(\C)\), define \[G^\vee := \{g^\vee\mid g\in G\}\subseteq\GL_4(\C).\]
Set \[s :=
\begin{pmatrix}
  0 & 0 & 0 & 1\\
  0 & 0 & -1 & 0\\
  0 & -1 & 0 & 0\\
  1 & 0 & 0 & 0
\end{pmatrix}\]
and define \(E(G) = \{ g^\vee, g^\vee s \mid g\in G\}\) for a group \(G\leq \GL_2(\C)\).
Then the groups \(E(G)\) with \(G\) in \ref{groups1} to \ref{groups4} are imprimitive symplectic reflection groups by \cite[Lemma 3.3]{Cohen80}.
In fact, all imprimitive but symplectically primitive symplectic reflection groups are conjugate to one of these groups by \cite[Theorem 3.6]{Cohen80}.

For any group \(G\leq\GL_2(\C)\) denote by \(\mathcal{R}(G)\) the set of reflection in \(G\).
For a group \(G\leq \GL_4(\C)\) denote by \(\mathcal{S}(G)\) the set of symplectic reflection in \(G\).
In what follows let \(G\) be one of the groups in \ref{groups1} to \ref{groups4} and write \(Z(G) = \mu_d\).
Let \(G_0\) be the largest complex reflection group contained in \(G\), as in Proposition \ref{prop:reflsub}.
We have \(G = \mu_d G_0\), but note that \(\mu_d\cap G_0 = Z(G_0)\).

\begin{lemma}
  \label{lem:Gnormal}
  The subgroups \(G^\vee\) and \(G_0^\vee\) are normal subgroups of \(E(G)\).
\end{lemma}
\begin{proof}
  For \(g, h\in G\) we have \(g^\vee h^\vee (g^\vee)^{-1} = (g h g^{-1})^\vee\in G^\vee\).
  If \(h \in G_0\), then also \(g^\vee h^\vee (g^\vee)^{-1}\in G_0^\vee\), since either \(g\in G_0\) or \(g\in Z(G)\).
  It remains to show \(sh^\vee s^{-1}\in G^\vee\) for \(h\in G\).
  Here, an easy calculation shows \(sh^\vee s^{-1} = \big((\det h)^{-1}h\big)^\vee \in G^\vee\) (see Lemma \ref{lem:dets}) and the same holds for \(h \in G_0\).
\end{proof}

\begin{lemma}
  \label{lem:Dnormal}
  The group \(D_d := \langle \mu_d^\vee, s\rangle\leq E(G)\) is the dihedral group of order \(2d\) and a normal subgroup of \(E(G)\).
\end{lemma}
\begin{proof}
  By definition, \(D_d\) is generated by \(r^\vee\) and \(s\), where \[r := \begin{pmatrix} \zeta_d & 0\\ 0 &\zeta_d\end{pmatrix},\] and the equalities \[(r^\vee)^d = s^2 = (sr^\vee)^2 = I_4\] hold, so \(D_d\) is indeed the dihedral group of order \(2d\).

  Let \(t\in E(G)\), so \(t = g^\vee s^k\) for some \(g\in G\) and \(k \in \{0, 1\}\).
  We have \(tr^\vee t^{-1} = (r^\vee)^{-1}\in D_d\), if \(k = 0\), and \(tr^\vee t^{-1} = r^\vee \in D_d\), if \(k = 1\).
  Further, we have \[t s t^{-1} = g^\vee s^k s s^{-k}(g^\vee)^{-1} = g^\vee s (g^\vee)^{-1},\] as \(s = s^{-1}\).
  But \[g^\vee s (g^\vee)^{-1} = \begin{pmatrix} 0 & A\\ A^{-1} & 0\end{pmatrix}\] with \[A := g\begin{pmatrix} 0 & 1\\ -1 & 0\end{pmatrix}g^\top = \begin{pmatrix} 0 & \det(g) \\ -\det(g) & 0\end{pmatrix}.\]
  By \(\det g = \zeta_d^l\) for some \(0\leq l < d\), it follows \(tst^{-1} = (r^l)^\vee s\in D_d\) and \(D_d\) is indeed a normal subgroup of \(E(G)\).
\end{proof}

\begin{proposition}
  \label{prop:refls}
  The group \(E(G)\) is a symplectic reflection group with symplectic reflections \[S := \mathcal{R}(G)^\vee\overset{.}{\cup}\{z^\vee s\mid z\in\mu_d\}.\]
\end{proposition}
\begin{proof}
  If \(g\in\mathcal{R}(G)\), so \(\rk(g - I_2) = 1\), then \(g^\vee\) is clearly a symplectic reflection.
  Also, for \(z = \left(\begin{smallmatrix} \zeta_d^k & 0\\0 & \zeta_d^k\end{smallmatrix}\right)\in\mu_d\) for some \(0\leq k < d\), we have
  \[z^\vee s =
    \begin{pmatrix}
      0 & 0 & 0 & \zeta_d^k\\
      0 & 0 & -\zeta_d^k & 0\\
      0 & -\zeta_d^{-k} & 0 & 0\\
      \zeta_d^{-k} & 0 & 0 & 0
    \end{pmatrix},\]
  so \(\rk(I_4 - z^\vee s) = 2\) and \(z^\vee s\) is a symplectic reflection.
  Hence all elements in \(S\) are indeed symplectic reflections and \(E(G)\) is a symplectic reflection group since \(E(G) = \langle S \rangle\).

  Let now \(t\in E(G)\) be a symplectic reflection.
  Then either \(t = g^\vee\) or \(t = g^\vee s\) for a \(g\in G\).
  In the first case, it directly follows \(g\in \mathcal{R}(G)\).
  So assume \(t = g^\vee s\).
  For ease of notation, we define \[A := g\begin{pmatrix} 0 & 1\\ -1 & 0 \end{pmatrix}\text{ and } B := (g^\top)^{-1}\begin{pmatrix}0 & -1\\ 1 & 0\end{pmatrix},\] so that \[ t = \begin{pmatrix} 0 & A\\ B & 0\end{pmatrix}.\]
  From \[I_4 - t = \begin{pmatrix} I_2 & A\\ B & I_2\end{pmatrix} = \begin{pmatrix} I_2 & 0\\B & I_2 - BA\end{pmatrix}\begin{pmatrix} I_2 & A\\ 0 & I_2\end{pmatrix}\]
  it follows that \(\rk(I_4 - t) = 2\) if and only if \(BA = I_2\), so \(A = B^{-1}\).
  An easy calculation shows that this requires \(g\) to be a scalar matrix, so \(g\in Z(G) = \mu_d\), as all scalar matrices lie in the centre of \(G\).
  Therefore all symplectic reflections in \(E(G)\) are elements of \(S\).

  Finally, note that the two given subsets of \(S\) contain matrices of different block-types, so their union is disjoint.
\end{proof}

\begin{corollary}
  \label{cor:symprefl}
  All symplectic reflections in \(E(G)\) lie either in \(G_0^\vee\) or in \(D_d\).
  None of the symplectic reflections of \(G_0^\vee\) is conjugate in \(E(G)\) to one of \(D_d\) and vice versa.
\end{corollary}
\begin{proof}
  The first part is clear since \(\mathcal{R}(G) = \mathcal{R}(G_0)\).
  The second part follows from Lemma \ref{lem:Gnormal} and Lemma \ref{lem:Dnormal}.
\end{proof}

We state for later reference:
\begin{lemma}
  \label{lem:sympconj}
  There are two \(D_d\)-conjugacy classes in \(\mathcal{S}(D_d)\), namely \([s]\) and \([(\zeta_dI_2)^\vee s]\).
  In case \(G\) belongs to \ref{groups1}, \ref{groups2}, or \ref{groups3} these are also the \(E(G)\)-conjugacy classes.
  In case \(G\) belongs to \ref{groups4}, there is only one \(E(G)\)-conjugacy class in \(\mathcal{S}(D_d)\).
\end{lemma}
\begin{proof}
  For the claim about \(D_d\)-conjugacy, see \cite[Section 8.3]{BellamyThiel16}.
  The computations in the proof of Lemma \ref{lem:Dnormal} show that for \(g\in E(G)\) we have \(gsg^{-1} = z^\vee s\) with \(z\in \{({\det h})I_2\mid h\in G\}\) (and for any such \(z\) there exists a \(g\in E(G)\)).
  Hence \(s\) and \((\zeta_d I_2)^\vee s\) are conjugate in \(E(G)\) if and only if there exists \(h\in G\) with \(\det h = \zeta_d\).
  By Lemma \ref{lem:dets}, this is the case if and only if \(G\) belongs to \ref{groups4}.
\end{proof}

\section{Symplectic reflection algebras}
Let again \(G\) be one of the groups in \ref{groups1} to \ref{groups4}.
Let \(G_0\) be the largest complex reflection group contained in \(G\) and let \(\mu_d = Z(G)\).
Let \(D_d := \langle \mu_d^\vee, s\rangle\leq E(G)\) as before.
Let \(V = \C^4\) with standard symplectic form \(\omega\) (notice that we already implicitly assumed this setting when we defined \(s\)).

We recall the definition of a symplectic reflection algebra as introduced in \cite{EtingofGinzburg02}.
For \(g\in\mathcal{S}(E(G))\) we have \(V = V^g \oplus (V^g)^\perp\) (orthogonal with respect to \(\omega\)).
Let \(\pi_g:V\to (V^g)^\perp\) be the projection and let \(\omega_g\) be the bilinear form defined by \(\omega_g(u, v) := \omega(\pi_g(u), \pi_g(v))\) for all \(u, v\in V\).
The symplectic reflection algebra of \((V, E(G))\) is defined to be \[\H_\c(V, E(G)) := T(V)\rtimes E(G) \Big/\Big\langle [u, v] - \sum_{g\in\mathcal{S}(E(G))} \c(g)\omega_g(u, v) g~\Big|~u, v\in V\Big\rangle,\] where \(\c :\mathcal{S}(E(G)) \to \C\) is an \(E(G)\)-conjugacy invariant function.
From now on we will omit the vector space in the notation and just write \(\H_\c(E(G))\) for this algebra.

We want to construct a simple module of \(\H_\c(E(G))\) of dimension strictly less than \(|{E(G)}|\) and then apply Theorem \ref{thm:EtingofGinzburgKaledin}.
To this end, we are going to deform a suitable module of an algebra \(\H_{\mathbf{c'}}(E(G))\) for a certain parameter \(\mathbf{c'}\).
To be able to state the precise result, we require a bit more notation. \\

By Corollary \ref{cor:symprefl}, we may split \(\c\) in two \(E(G)\)-invariant functions \(\c_1:\mathcal{S}(E(G))\to \C\) and \(\c_2:\mathcal{S}(E(G))\to \C\) given by \[\c_1|_{\mathcal{S}(G_0^\vee)} = \c|_{\mathcal{S}(G_0^\vee)}\text{ and }\c_1|_{\mathcal{S}(D_d)} = 0\text{ resp.\ }\c_2|_{\mathcal{S}(G_0^\vee)} = 0\text{ and }\c_2|_{\mathcal{S}(D_d)} = \c|_{\mathcal{S}(D_d)},\] so we may think of \(\c\) as \(\c_1 + \c_2\).
By abuse of notation, we also write \(\c_1\) resp.\ \(\c_2\) for the restrictions \(\c_1|_{\mathcal{S}(G_0^\vee)}\) resp.\ \(\c_2|_{\mathcal{S}(D_d)}\).

We may consider the symplectic reflection algebras \(\H_{\c_1}(G_0)\) and \(\H_{\c_1}(G)\) (or more precisely \(\H_{\c_1}(G_0^\vee)\) and \(\H_{\c_1}(G^\vee)\)) with the embeddings \(\H_{\c_1}(G_0) \subseteq \H_{\c_1}(G) \subseteq \H_{\c_1}(E(G))\).
Notice however that \(\c_1\) is in general \emph{not} a generic (or even arbitrary) parameter for \(\H_{\c_1}(G_0)\), since \(G_0^\vee\)-invariant functions are not necessarily \(E(G)\)-invariant.

Let \(\chi_0,\dots,\chi_{d - 1}\) be the irreducible characters of \(Z(G) = \mu_d\), ordered such that \[\chi_l\left(\begin{pmatrix} \zeta_d^k & 0\\ 0 & \zeta_d^k\end{pmatrix}\right) = \zeta_d^{kl}\] for all \(0\leq k, l< d\) and a primitive \(d\)-th root of unity \(\zeta_d\).

Notice that \(d\) is even as \(-I_2\in Z(G)\).
We label the irreducible representations of \(D_d\) as follows.
There are four 1-dimensional representations \(\operatorname{Triv}\), \(\operatorname{Sgn}\), \(V_1\) and \(V_2\), where \(\operatorname{Triv}|_{Z(G)^\vee} = \operatorname{Sgn}|_{Z(G)^\vee} = \chi_0\) and \[V_1|_{Z(G)^\vee} = V_2|_{Z(G)^\vee} = \chi_{\frac{d}{2}},\]
(note that \(Z(G)^\vee\leq D_d\)).
Further, there are the 2-dimensional representations \(\phi_1,\dots,\phi_{\frac{d}{2} - 1}\) for which we have \[\phi_i|_{Z(G)^\vee} = \chi_i \oplus \chi_{d - i}.\]
See \cite[Section 8.2]{BellamyThiel16} for more details and precise definitions of these representations.

We say an irreducible representation \(\phi\) of \(D_d\) is \(\c_2\)-\emph{rigid}, if \(\phi\) is (isomorphic to) a simple \(\H_{\c_2}(D_d)\)-module, see \cite{BellamyThiel16} for details.
The following proposition reduces the problem of constructing \(\H_\c(E(G))\)-modules to constructing \(\H_{\c_1}(G)\)-modules.

\begin{proposition}
  \label{prop:r}
  Let \(L\) be a simple \(\H_{\c_1}(G)\)-module and set \[ R(L) := \H_{\c_1}(E(G))\otimes_{\H_{\c_1}(G)}L . \]
  Then \(R(L)\) is a \(\H_\c(E(G))\)-module if and only if all constituents of \(R(L)|_{D_d}\) are \(\c_2\)-rigid.
\end{proposition}
\begin{proof}
  By definition, \(R(L)\) is an \(\H_{\c_1}(E(G))\)-module.
  We just need to show that it naturally deforms to a \(\H_\c(E(G))\)-module.
  The defining relations for \(\H_\c(E(G))\) are \[[u, v] = \sum_{g\in\mathcal{S}(G^\vee)}\c_1(g)\omega_g(u, v)g + \sum_{g\in\mathcal{S}(D_d)}\c_2(g)\omega_g(u, v)g\] in contrast to \[[u, v] = \sum_{g\in\mathcal{S}(G^\vee)}\c_1(g)\omega_g(u, v)g\] for \(\H_{\c_1}(E(G))\).
  As \(R(L)\) is an \(\H_{\c_1}(E(G))\)-module this means that \([u, v]\) acts as \[\sum_{g\in\mathcal{S}(G^\vee)}\c_1(g)\omega_g(u, v)g.\]
  Hence \(R(L)\) is an \(\H_\c(E(G))\)-module if and only if \(\sum_{g\in\mathcal{S}(D_d)}\c_2(g)\omega_g(u, v)g\) acts as zero on \(R(L)\) for all \(u, v\in V\) that is, if and only if \[\sum_{g\in \mathcal{S}(D_d)}\c_2(g)\omega_g(u, v)\phi(g) = 0\] for any constituent \(\phi\) of \(R(L)|_{D_d}\).
  By \cite[Lemma 4.10]{BellamyThiel16}, this holds if and only if all constituents of \(R(L)|_{D_d}\) are \(\c_2\)-rigid.
\end{proof}

\begin{lemma}
  \label{lem:rigids}
  An irreducible representation \(\phi\in\Irr D_d\) is \(\c_2\)-rigid for all \(E(G)\)-invariant functions \(\c_2:\mathcal S(D_d)\to \C\) if and only if:
  \begin{enumerate}[(a)]
    \item \(\phi = \phi_i\) for an \(1 < i < (d - 2)/2\), in case \(G\) belongs to \ref{groups1}, \ref{groups2}, or \ref{groups3},
    \item \(\phi = \phi_i\) for an \(1 < i \leq (d - 2)/2\) or \(\phi \in \{V_1,V_2\}\), in case \(G\) belongs to \ref{groups4}.
  \end{enumerate}
\end{lemma}
\begin{proof}
  By \cite[Proposition 8.3]{BellamyThiel16}, the representations \(\phi_i\) for \(1 < i < (d-2)/2\) are \(\c_2\)-rigid for arbitrary parameters \(\c_2\).
  By Lemma \ref{lem:sympconj}, the function \(\c_2\) is determined by its values at \(s\) and \((\zeta_dI_2)^\vee s\).
  \begin{enumerate}[(a),leftmargin=0pt,itemindent=*]
    \item The symplectic reflections \(s\) and \((\zeta_dI_2)^\vee s\) are not \(E(G)\)-conjugate by Lemma \ref{lem:sympconj}.
      Hence there exist parameters \(\c_2\) with \(\c_2(s) \neq \c_2((\zeta_dI_2)^\vee s)\) and \(\c_2(s) \neq -\c_2((\zeta_dI_2)^\vee s)\) and all other representations are not \(\c_2\)-rigid for those parameters by \cite[Proposition 8.3]{BellamyThiel16}.
    \item Here, Lemma \ref{lem:sympconj} states that there is only one \(E(G)\)-conjugacy class in \(\mathcal{S}(D_d)\).
      Therefore all parameters fulfil \(\c_2(s) = \c_2((\zeta_dI_2)^\vee s)\) and only \(\phi_1\), \(\operatorname{Triv}\), and \(\operatorname{Sgn}\) are not \(\c_2\)-rigid by \cite[Proposition 8.3]{BellamyThiel16}.\qedhere
  \end{enumerate}
\end{proof}

\begin{corollary}
  \label{cor:rigids}
  Let \(\phi\) be any representation of \(D_d\).
  Then all constituents of \(\phi\) are \(\c_2\)-rigid for all \(E(G)\)-invariant functions \(\c_2:\mathcal S(D_d)\to \C\) if and only if:
  \begin{enumerate}[(a)]
    \item \(\chi_i\mid \phi|_{Z(G)}\) implies \(i\notin\{0, 1, \frac{d}{2} - 1, \frac{d}{2}, \frac{d}{2} + 1, d - 1\}\) in case \(G\) belongs to \ref{groups1}, \ref{groups2}, or \ref{groups3},
    \item \(\chi_i\mid \phi|_{Z(G)}\) implies \(i\notin\{0, 1, d - 1\}\) in case \(G\) belongs to \ref{groups4}.
  \end{enumerate}
\end{corollary}

The action of \(G_0\) resp.\ \(G\) on \(V\) leaves a Lagrangian subspace \(\h\) invariant and we may identify \(\h\) with the reflection representation of \(G_0\).
Then \(\h = \C^2\) and \(\zeta_d I_2\in \mu_d\) acts as the scalar \(\zeta_d\) on \(\h\) and as \(\zeta_d^{-1}\) on \(\h^\ast\).
We may write \(V = \h \oplus \h^\ast\) (but this decomposition is of course not stable under the action of \(s\)).
Then we can define a \(\Z\)-grading on \(\H_{\c_1}(G_0)\) by putting \(\h^\ast\) in degree 1, \(\h\) in degree \(-1\) and \(G_0\) in degree 0.
In the same way, we obtain a \(\Z\)-grading on \(\H_{\c_1}(G)\) and the inclusion \(\H_{\c_1}(G_0)\subseteq \H_{\c_1}(G)\) preserves this grading.

Set \[\overline{\H}_{\c_1}(G_0) := \H_{\c_1}(G_0)\big/\big(\C[\h]^{G_0}\otimes \C[\h^\ast]^{G_0}\big)_+\H_{\c_1}(G_0).\]
This algebra has a triangular decomposition \[\overline{\H}_{\c_1}(G_0) \cong \C[\h]^{\co G_0}\otimes \C G_0 \otimes \C[\h^\ast]^{\co G_0},\] where \(\C[\h]^{\co G_0} := \C[\h]/\C[\h]^{G_0}_+\C[\h]\), see \cite[Corollary 2.1]{Thiel17}.
Given \(\lambda\in\Irr G_0\), we then have the baby Verma module \[\Delta(\lambda) := \overline{\H}_{\c_1}(G_0)\otimes_{\C[\h^\ast]^{\co G_0}\rtimes \C G_0}\lambda\] of \(\overline{\H}_{\c_1}(G_0)\) corresponding to \(G_0\) as in \cite{Thiel17}.
The module \(\Delta(\lambda)\) has a simple head \(L(\lambda)\) by \cite[Theorem 2.3]{Thiel17}.
We may consider both of them as \(\H_{\c_1}(G_0)\)-modules by letting \(\H_{\c_1}(G_0)\) act via the quotient morphism \(\H_{\c_1}(G_0) \twoheadrightarrow \overline{\H}_{\c_1}(G_0)\).
Notice that \(L(\lambda)\) is also simple as \(\H_{\c_1}(G_0)\)-module.

\begin{lemma}
  \label{lem:simple}
  Let \(\lambda\in \Irr G\). Then
  \begin{enumerate}[(i)]
    \item \(\lambda|_{G_0}\in\Irr G_0\) and
    \item The $H_{\c_1}(G_0)$-module structure on any graded quotient of $\Delta(\lambda|_{G_0})$ extends to $H_{\c_1}(G)$. In particular, $L(\lambda|_{G_0})$ is a graded (simple) $H_{\c_1}(G)$-module. 
  \end{enumerate}
\end{lemma}
\begin{proof} \hfill
  \begin{enumerate}[(i),leftmargin=0pt,itemindent=*,]
    \item This is \cite[Theorem III.2.14]{Feit82} since \(G/G_0\) is cyclic.
    \item We have to define an action of \(Z(G)\) on \(\Delta(\lambda|_{G_0})\).
      By \cite[Lemma 2.5]{Thiel17}, we have \[\Delta(\lambda|_{G_0}) \cong \C[\h]^{\co G_0}\otimes_\C \lambda|_{G_0}\] as vector spaces, in particular \(\Delta(\lambda|_{G_0})\) is concentrated in non-negative degree.
      Let \(Z(G)\) act by \(\chi\) on \(\lambda\).
      By the above, \(\zeta_d I_2 \in Z(G)=\mu_d\) acts by \(\zeta_d^{-1}\) on \(\h^\ast\).
      We obtain an action of \(Z(G)\) on \(\Delta(\lambda|_{G_0})_k\) for any \(k\geq 0\) by letting \(\zeta_d I_2\) act by \(\zeta_d^{-k}\) on \(\C[\h]^{\co G_0}_k\) and by \(\chi(\zeta_d I_2)\) on \(\lambda|_{G_0}\).
      Then this action of \(Z(G)\) extends \(\Delta(\lambda|_{G_0})\) to a module over \(\H_{\c_1}(G)\).

      Now let \(M\leq \Delta(\lambda|_{G_0})\) be any graded \(\H_{\c_1}(G_0)\)-submodule.
      Since \(M\) is graded, it is stable under the action of \(\C^\times\) induced by the action of \(\C^\times\) on \(\h\).
      The given action of \(Z(G)\) on \(\h^\ast\) is just a restriction of this action to the subgroup \(\langle \zeta_d\rangle\leq \C^\times\).
      Hence this also extends \(M\) to an \(\H_{\c_1}(G)\)-module.

      As \(L(\lambda|_{G_0})\) is a graded quotient of \(\Delta(\lambda|_{G_0})\), this turns \(L(\lambda|_{G_0})\) into an \(\H_{\c_1}(G)\)-module too and \(L(\lambda|_{G_0})\) is of course simple as such a module. \qedhere
  \end{enumerate}
\end{proof}

\begin{theorem}
  \label{thm:nores}
  If there exists \(\lambda\in \Irr G\) such that \(L(\lambda|_{G_0})|_{D_d}\) is \(\c_2\)-rigid for all \(E(G)\)-invariant functions \(\c_2:\mathcal S(D_d)\to \C\) and \(\dim L(\lambda|_{G_0}) < |{G}|\), then \(\C^4/E(G)\) does not admit a symplectic resolution.
\end{theorem}
\begin{proof}
  Since \(L(\lambda|_{G_0})\) fulfils the conditions of Proposition \ref{prop:r}, we obtain an \(\H_\c(E(G))\)-module \(R(L(\lambda|_{G_0}))\).
  By construction, we have \[\dim R(L(\lambda|_{G_0})) = \dim\big(\H_{\c_1}(E(G))\otimes_{\H_{\c_1}(G)}L(\lambda|_{G_0})\big) = 2\dim L(\lambda|_{G_0}) < |{E(G)}|,\] since \(|{E(G)}| = 2|{G}|\).
  Then any simple quotient \(L\) of \(R(L(\lambda|_{G_0}))\) will also fulfil \(\dim L < |{E(G)}|\).
  As this holds for arbitrary parameters \(\c\), if follows that the variety \(\C^4/E(G)\) does not admit a symplectic resolution by Theorem \ref{thm:EtingofGinzburgKaledin}.
\end{proof}

We use a crude first estimate to show that all but finitely many of the groups on the list admit a simple module as in Theorem \ref{thm:nores}.
Let \(N := |{\mathcal{S}(G_0^\vee)}|\) be the number of symplectic reflections in \(G_0\).
The coinvariant ring \(\C[\h]^{\operatorname{co}G_0}\) is a (positively) graded ring with \((\C[\h]^{\operatorname{co}G_0})_k = 0\) for \(k > N\), by \cite[Proposition 20-3A]{Kane01}.
This implies \(\Delta(\lambda)_k = 0\) for each \(k > N\) or \(k < 0\) and any simple \(G_0\)-module \(\lambda\).

\begin{proposition}
  \label{prop:exsimple}
  The group \(G\) admits a simple module \(\lambda\) as in Theorem \ref{thm:nores} if \(G_0\lneq G\) and
  \begin{enumerate}[(a)]
    \item\label{prop:exsimple:a} \(2N + 6 < d\) in case \(G\) belongs to \ref{groups1}, \ref{groups2}, or \ref{groups3},
    \item\label{prop:exsimple:b} \(N + 3 < d\) in case \(G\) belongs to \ref{groups4}.
  \end{enumerate}
\end{proposition}
\begin{proof}
  We only prove \ref{prop:exsimple:a}, \ref{prop:exsimple:b} follows analogously.
  Note that \(d - 2\geq 0\) by assumption.
  Let \(\lambda\in\Irr G\) be any irreducible summand of \(\Ind_{Z(G)}^G\chi_{d - 2}\), so \(\lambda\) restricts to a multiple of \(\chi_{d - 2}\) on \(Z(G)\).
  As in the proof of Lemma \ref{lem:simple}, \(Z(G)\) acts on \(\Delta(\lambda|_{G_0})_k\) by \[\chi_{d - k}\otimes \chi_{d - 2} = \chi_{d - k - 2},\] for \(k \geq 0\).
  Since \(L(\lambda|_{G_0})\) is a quotient of \(\Delta(\lambda|_{G_0})\), this implies that if \[[L(\lambda|_{G_0})|_{Z(G)}:\chi_i]\neq 0\] then \(i\in\{d - N - 2, d - N - 1, \dots, d - 2\}\).
  Hence if \(2N + 6 < d\), then \(N + 2 < \frac{d - 2}{2}\), so \[d - N - 2 = d - (N + 2) > d - \frac{d - 2}{2} = \frac{d}{2} + 1.\]
  Then \(L(\lambda|_{G_0})|_{D_d}\) is \(\c_2\)-rigid for all \(\c_2\) by Corollary \ref{cor:rigids}.

  We have \(\dim L(\lambda|_{G_0}) \leq |{G_0}|\) by \cite[Theorem 1.7]{EtingofGinzburg02}, hence \(\dim L(\lambda|_{G_0}) < |{G}|\) since \(G_0 \lneq G\).
\end{proof}

\section{Sharp bounds}

In Table \ref{tab:numref} we recall the number of reflections in the possible groups \(G_0\) from \cite{Cohen76} together with the minimal value of \(d\) fulfilling the condition in Proposition \ref{prop:exsimple} (which does not mean that there exists a group \(G\) for such a \(d\)).
This gives the groups \(G\), for which Proposition \ref{prop:exsimple} does not apply, as in Table \ref{tab:open}.
Using data computed with \textsc{Champ} \cite{Thiel15}, we want to find all groups which fulfil the assumptions of Theorem \ref{thm:nores}.
We describe the necessary computations and give a concrete example below.

As before, let \(G_0\) be one of the complex reflection groups from Table \ref{tab:comprefl} and let \((G_d)_{d\in\mathcal{D}}\) be the family of supergroups containing \(G_0\) as subgroup generated by the reflections for a set of indices \(\mathcal{D}\) determined by the conditions in \ref{groups1} to \ref{groups4} and Proposition \ref{prop:reflsub}.
Let \(\lambda\in\Irr G_0\) and let \(Z(G_0) = \langle\zeta\rangle\).
Then \(\lambda(\zeta) = \zeta_l I_{\dim\lambda}\) for a certain primitive \(l\)-th root of unity \(\zeta_l\) with \(l\mid d_0\).
Hence we can extend \(\lambda\) to a representation \(\lambda_d\) of \(G_d\) for any \(d\in\mathcal{D}\) by setting \(\lambda_d|_{G_0} = \lambda\) and \(\lambda_d(\eta) = \zeta_{l'}I_{\dim\lambda}\), where \(Z(G) = \langle\eta\rangle\) and \(l' = l\frac{d}{d_0}\) (note that \(l'\mid d\), since \(l\mid d_0\) and \(d_0\frac{d}{d_0} = d\)).
Here, \(\zeta_{l'}\) is a primitive \(l'\)-th root of unity with \(\zeta_{l'}^{d/d_0} = \zeta_l\).
In particular, there may exist more than one choice for \(\lambda_d\).

Now one can find, if it exists, the smallest \(d_1\in\mathcal D\) such that \(\lambda_{d_1}(\eta) = \eta^{-m}I_{\dim\lambda}\) with \(2\leq m < \frac{d_1}{2} - 1\) respectively \(2\leq m < d_1 - 1\) if \(G_0\) belongs to \ref{groups4}.

Let \(k\geq 0\) be minimal such that \(L(\lambda)_k = 0\) with respect to all parameters \(\c_1\), which we can compute using \textsc{Champ}.
Then the claim of Proposition \ref{prop:exsimple} also holds for all \(d\in\mathcal D\) with \(d\geq d_1\) and \(d - (k - 1) - m > \frac{d}{2} + 1\) respectively \(d - (k - 1) - m > 1\) if \(G_0\) belongs to \ref{groups4}.

We give the results of our computations and in particular the best possible values for \(k\) and \(m\) for each of the families of groups in Table \ref{tab:champres}.
Using those bounds for \(d\), we obtain a ``better'' version of Table \ref{tab:open}, see Table \ref{tab:open2}.
However, this also means that for the groups in Table \ref{tab:open2} (besides those, for which we could not do any computations), there does not exist any simple module \(\lambda\) fulfilling the conditions of Theorem \ref{thm:nores}.

\begin{example}
  We carry out the described computations for the group \(G_0 := \mu_6 \T\).
  The family of supergroups is given by \(G_d := \mu_d \T\) for \(d = 12a + 6\) with \(a \in \Z_{\geq 0}\).
  Let \(\omega\in \C\) be a primitive third root of unity and set \(\zeta_6 := -\omega^{-1}\).
  Then we may choose the matrix \(\zeta := \zeta_6I_2\) as generator for \(Z(G_0) = \mu_6\).
  Going through the representations of \(G_0\) in the database of \textsc{Champ}, we see that representation numbered 19 with character \(\phi_{3,4}\) maps \(\zeta\) to \((-\omega - 1)I_3 = \zeta_6^{-2}I_3\).
  In the above notation, we hence have \(m = 2\).
  Note that this is the ``best possible" value of \(m\) since we require \(m \geq 2\).

  This gives the lower bound \(m = 2 < \frac{d_1}{2} - 1\), so that \(d_1 > 6\), that is, \(d_1 = 18 = 3d_0\).
  Using \textsc{Champ}, we see that the top degree of \(L(\lambda)\) is 4, hence we have \(k = 5\).
  Therefore we have the additional restriction \[d - (k - 1) - m = d - 6 > \frac{d}{2} + 1,\] which simplifies to \(d > 14\).
  In conclusion, we improved the lower bound for \(d\) in Proposition \ref{prop:exsimple} to \(d \geq 18\), leaving only the group \(G_0\) itself.
\end{example}

\section{The group \texorpdfstring{\(W(S_2)\)}{W(S2)}}
\label{sec:onemoregroup}

In this section we show that for the group \(W(S_2)\) of \cite[Table III]{Cohen80} there is no symplectic resolution of the corresponding linear quotient.
This group is one of the few symplectically and complex primitive groups; we follow the same strategy used in \cite{BellamySchedler16} to treat these groups.
Namely, we are going to show, or rather compute that there is a subgroup of \(W(S_2)\), say \(H\), which is the stabilizer of a vector.
We can identify \(H\) with the improper symplectic group coming from the complex reflection group \(G(4, 4, 3)\) in the classification by Shephard and Todd \cite{ShephardTodd54}.
Since the corresponding linear quotient of this group does not have a symplectic resolution by \cite{Bellamy09}, the same holds for the quotient by \(W(S_2)\) by a result of Kaledin \cite[Theorem 1.6]{Kaledin03}.

The computer calculations leading to the result we are going to present were carried out and cross-checked using the software package Hecke \cite{Hecke} and the computer algebra systems GAP \cite{Gap} and Magma \cite{Magma}.

\subsection{The group}

The group of interest here is a subgroup of \(\Sp_8(\C)\) of order \(2^{10}3^4 = 82944\).
Like all symplectically and complex primitive groups, it is given by a root system in \cite[Table II]{Cohen80}.
Cohen gives 72 root lines for the group, however, already four are enough to generate a group of the correct order.
The respective root lines are
\begin{align*}
  &( 1, i, 0, 0, 0, 0, 1, -i), &&( 1 - i, 1 - i, 0, 0, 0, 0, 0, 0),\\
  &( 1 - i, 0, 1 - i, 0, 0, 0, 0, 0), &&( 2, 0, 0, 0, 0, 0, 0, 0),
\end{align*}
in \(\C^8\).
Note that these are the ``complexified'' versions of the vectors over the quaternions given in \cite{Cohen80}.
The group \(W(S_2)\leq \Sp_8(\C)\) is now generated by the symplectic reflection matrices
\begin{align*}
  M_1 := \frac{1}{2}&\left(\begin{smallmatrix}
     1 & i &    &    &   &    & -1 & -i \\
    -i & 1 &    &    &   &    & -i &  1 \\
       &   &  1 &  i & 1 &  i &    &    \\
       &   & -i &  1 & i & -1 &    &    \\
       &   &  1 & -i & 1 & -i &    &    \\
       &   & -i & -1 & i &  1 &    &    \\
    -1 & i &    &    &   &    &  1 & -i \\
     i & 1 &    &    &   &    &  i &  1
  \end{smallmatrix}\right),
  &&M_2 := \left(\begin{smallmatrix}
       & -1 &    &    &    &    &    &    \\
    -1 &    &    &    &    &    &    &    \\
       &    &  1 &    &    &    &    &    \\
       &    &    &  1 &    &    &    &    \\
       &    &    &    &    & -1 &    &    \\
       &    &    &    & -1 &    &    &    \\
       &    &    &    &    &    &  1 &    \\
       &    &    &    &    &    &    &  1
  \end{smallmatrix}\right),
\end{align*}
\begin{align*}
  M_3 := &\left(\begin{smallmatrix}
       &    & -1 &    &    &    &    &    \\
       &  1 &    &    &    &    &    &    \\
    -1 &    &    &    &    &    &    &    \\
       &    &    &  1 &    &    &    &    \\
       &    &    &    &    &    & -1 &    \\
       &    &    &    &    &  1 &    &    \\
       &    &    &    & -1 &    &    &    \\
       &    &    &    &    &    &    &  1
  \end{smallmatrix}\right),
  &&M_4 := \left(\begin{smallmatrix}
    -1 &    &    &    &    &    &    &    \\
       &  1 &    &    &    &    &    &    \\
       &    &  1 &    &    &    &    &    \\
       &    &    &  1 &    &    &    &    \\
       &    &    &    & -1 &    &    &    \\
       &    &    &    &    &  1 &    &    \\
       &    &    &    &    &    &  1 &    \\
       &    &    &    &    &    &    &  1
  \end{smallmatrix}\right),
\end{align*}
which one obtains from these root lines, see \cite{Cohen80} for details.

\subsection{The subgroup}

Let \(v := (0, 0, 1, 0, 0, 0, 0, -1)^\top\in \C^8\) and let \(H\leq W(S_2)\) be the stabilizer of \(v\) with respect to the natural action of \(W(S_2)\) on \(\C^8\).
Using the command \texttt{Stabilizer} in either GAP \cite{Gap} or Magma \cite{Magma} one can compute this group: \[H = \langle M_2, M_4, M_1M_3M_4M_2M_4M_3M_1 \rangle.\]
The space \(V^H \leq \C^8\) of vectors fixed by \(H\) is generated by \(v\) and \((0, 0, 0, 1, 0, 0, 1, 0)^\top\).
Its \(H\)-invariant complement \(W\) is then generated by the columns \(w_1,\dots, w_6\in \C^8\) of the matrix
\[
\frac{1}{2}
\left(\begin{smallmatrix}
  -\zeta^3 & -\zeta^3 &          & -\zeta^3 & -\zeta^3 &          \\
  -\zeta   &  \zeta   &          &  \zeta   & -\zeta   &          \\
           &          & -\zeta^3 &          &          &  \zeta^3 \\
           &          & -\zeta   &          &          & -\zeta   \\
  -\zeta   & -\zeta   &          &  \zeta   &  \zeta   &          \\
   \zeta^3 & -\zeta^3 &          &  \zeta^3 & -\zeta^3 &          \\
           &          &  \zeta   &          &          &  \zeta   \\
           &          & -\zeta^3 &          &          &  \zeta^3
\end{smallmatrix}\right),
\]
where \(\zeta\in \C\) is a primitive 8-th root of unity such that \(\zeta^2 = i\).

By changing the basis from \(\C^8\) to \(W \oplus V^H\) and restricting to \(W\) we may identify \(H\) with a subgroup \(H_W\) of \(\Sp(W)\) generated by the matrices
\[\left(\begin{smallmatrix}
    & -i &   &    &   &   \\
  i &    &   &    &   &   \\
    &    & 1 &    &   &   \\
    &    &   &    & i &   \\
    &    &   & -i &   &   \\
    &    &   &    &   & 1
  \end{smallmatrix}\right),
  \left(\begin{smallmatrix}
       & -1 &   &    &    &   \\
    -1 &    &   &    &    &   \\
       &    & 1 &    &    &   \\
       &    &   &    & -1 &   \\
       &    &   & -1 &    &   \\
       &    &   &    &    & 1
  \end{smallmatrix}\right),
  \left(\begin{smallmatrix}
       &   & -i &    &   &   \\
       & 1 &    &    &   &   \\
     i &   &    &    &   &   \\
       &   &    &    &   & i \\
       &   &    &    & 1 &   \\
       &   &    & -i &   &
  \end{smallmatrix}\right).
  \]
The basis of \(W\) was chosen, so that the symplectic form on \(\Sp(W)\) is given by the matrix
\[
  \begin{pmatrix} & I_3 \\ -I_3 & \end{pmatrix}.
\]
One can directly see that \(H_W\) leaves the subspace \(\langle w_1, w_2, w_3\rangle\) invariant and that this is a Lagrangian subspace.
Hence \(H_W\) is an improper group and can be identified with a complex reflection group \(G \leq \GL(\langle w_1, w_2, w_3\rangle)\).
Since this group is of rank 3 and order 96 it must be conjugate to \(G(4, 4, 3)\) in the classification \cite{ShephardTodd54}.

\subsection{Conclusion}

\begin{theorem}
  The linear quotient \(\C^8/W(S_2)\) by the symplectic reflection group \(W(S_2)\) as given in \cite[Table III]{Cohen80} does not admit a symplectic resolution.
\end{theorem}
\begin{proof}
  Assume there does exist such a resolution.
  Let \(v\in \C^8\) be any vector, \(G_v\leq W(S_2)\) the stabilizer of this vector and \(V\leq C^8\) the \(G_v\)-invariant complement of the subspace \((\C^8)^{G_v}\) of vectors fixed by \(G_v\).
  Then also \(V/G_v\) admits a symplectic resolution by \cite[Theorem 1.6]{Kaledin03}.
  However, by the calculations above there exists a vector \(v\in \C^8\) such that \(G_v\) acts on the invariant complement of the fixed space as \(G(4, 4, 3)\).
  Hence the quotient by \(G_v\) does not admit a symplectic resolution by \cite[Corollary 1.2]{Bellamy09} and the same must hold for \(\C^8/W(S_2)\).
\end{proof}

\begin{remark}
  One might want to use the same approach as presented in this section for the remaining groups.
  However, if the group is of dimension 4 any non-trivial subgroup stabilizing a vector is of dimension 2, so that the corresponding quotient by the subgroup always admits a symplectic resolution.

  This leaves only the groups \(W(S_1)\), \(W(R)\) and \(W(U)\), which are of dimension 6, 8 and 10 respectively.
  For the group \(W(S_1)\) we could not find any suitable subgroup, and the groups \(W(R)\) and \(W(U)\) are too large so that an exhaustive search for subgroups is not feasible.
\end{remark}

\clearpage
\section{Tables}
\begin{table}[htbp]
  \centering
  \begin{tabular}{c | c | c || c | c | c}
    \multirow{2}{*}{Group} & Number of & Minimal &
    \multirow{2}{*}{Group} & Number of & Minimal \\
    & reflections & value of \(d\) &
    & reflections & value of \(d\) \\
    \hline
    \hline
    \(\mu_6\T\) & 16 & 39 &
    \(\mu_{12}\T\) & 22 & 51 \\
    \hline
    \(\mu_4\O\) & 18 & 43 &
    \(\mu_8\O\) & 30 & 67 \\
    \(\mu_{12}\O\) & 34 & 75 &
    \(\mu_{24}\O\) & 46 & 99 \\
    \hline
    \(\mu_4\I\) & 30 & 67 &
    \(\mu_6\I\) & 40 & 87 \\
    \(\mu_{10}\I\) & 48 & 103 &
    \(\mu_{12}\I\) & 70 & 147 \\
    \(\mu_{20}\I\) & 78 & 163 &
    \(\mu_{30}\I\) & 88 & 183 \\
    \(\mu_{60}\I\) & 118 & 243 &
    & & \\
    \hline
    \(\OT_2\) & 12 & 16 &
    \(\OT_4\) & 18 & 22 \\
    \(\OT_6\) & 28 & 32 &
    \(\OT_{12}\) & 34 & 38 \\
  \end{tabular}
  \caption{\label{tab:numref}Number of reflections in the groups \(G_0\)}
\end{table}

\begin{table}[htbp]
  \centering
  \begin{tabular}{c | l || c | l}
    \multirow{2}{*}{\(G_0\)} & Groups containing \(G_0\) as & \multirow{2}{*}{\(G_0\)} & Groups containing \(G_0\) as \\
    & largest reflection group & & largest reflection group\\
    \hline
    \hline
    \(\mu_6\T\) & \(\mu_d \T\), \(d \in\{6, 18, 30\}\) &
    \(\mu_{12}\T\) & \(\mu_d \T\), \(d \in \{12, 24, 36, 48\}\)\\
    \hline
    \(\mu_4\O\) & \(\mu_d \O\), \(d \in \{4, 20, 28\}\) &
    \(\mu_8\O\) & \(\mu_d \O\), \(d \in \{8, 16, 32, 40, 56, 64\}\)\\
    \(\mu_{12}\O\) & \(\mu_d \O\), \(d \in \{12, 36, 60\}\) &
    \(\mu_{24}\O\) & \(\mu_d \O\), \(d \in \{24, 48, 72, 96\}\)\\
    \hline
    \(\mu_4\I\) & \(\mu_d \I\), \(d \in \{4, 8, 16, 28, 32, 44, 52, 56, 64\}\) &
    \(\mu_6\I\) & \(\mu_d \I\), \(d \in \{6, 18, 42, 54, 66, 78\}\)\\
    \(\mu_{10}\I\) & \(\mu_d \I\), \(d \in \{10, 50, 70\}\) &
    \(\mu_{12}\I\) & \(\mu_d \I\), \(d \in \{12, 24, 36, 48, 72, 84,\)\\
    \(\mu_{20}\I\) & \(\mu_d \I\), \(d \in \{20, 40, 80, 100, 140\}\) & & \phantom{\(\mu_d\I\), \(d\in\{\)}\(96, 108, 132, 144\}\)\\
    \(\mu_{30}\I\) & \(\mu_d \I\), \(d \in \{30, 90, 150\}\) &
    \(\mu_{60}\I\) & \(\mu_d \I\), \(d \in \{60, 120, 180, 240\}\)\\
    \hline
    \(\OT_2\) & \(\OT_d\), \(d\in \{2, 10, 14\}\) &
    \(\OT_4\) & \(\OT_d\), \(d \in \{4, 20\}\)\\
    \(\OT_6\) & \(\OT_d\), \(d \in \{6, 18, 30\}\) &
    \(\OT_{12}\) & \(\OT_d\), \(d \in \{12, 36\}\)\\
  \end{tabular}
  \caption{\label{tab:open}Groups for which Proposition \ref{prop:exsimple} does not apply}
\end{table}


\begin{table}[htbp]
  \centering
  \begin{tabular}{c | c | c | c | c | c | c | c }
    \multirow{2}{*}{\(G_0\)} & Shephard-  & Character         & Number            & \multirow{2}{*}{\(k\)} & \multirow{2}{*}{\(d_1\)} & \multirow{2}{*}{\(m\)} & Lower          \\
                             & Todd       & of \(\lambda\)    & in \textsc{Champ} &                        &                          &                        & bound of \(d\) \\
    \hline
    \(\mu_6\T\)              & \(G_5\)    & \(\phi_{3,4}\)    & 19                & 5                      & \(3d_0\)                 & 2                      & 15             \\
    \(\mu_{12}\T\)           & \(G_7\)    & \(\phi_{3,10}\)   & 37                & 7                      & \(d_0\)                  & 2                      & 19             \\
    \hline
    \(\mu_4\O\)              & \(G_{13}\) & \(\phi_{2,1}\)    & 7                 & 3                      & \(5d_0\)                 & 3                      & 11             \\
    \(\mu_8\O\)              & \(G_9\)    & \(\phi_{4,5}\)    & 32                & 7                      & \(2d_0\)                 & 3                      & 21             \\
    \(\mu_{12}\O\)           & \(G_{15}\) & \(\phi_{3,10}''\) & 36                & 11                     & \(d_0\)                  & 2                      & 27             \\
    \(\mu_{24}\O\)           & \(G_{11}\) & \multicolumn{6}{l}{\textit{No data available.}} \\
    \hline
    \(\mu_{4}\I\)            & \(G_{22}\) & \(\phi_{4,6}\)    & 12                & 1                      & \(2d_0\)                 & 2                      & 7              \\
    \(\mu_{6}\I\)            & \(G_{20}\) & \(\phi_{3,10}'\)  & 13                & 1                      & \(3d_0\)                 & 2                      & 7              \\
    \(\mu_{10}\I\)           & \(G_{16}\) & \(\phi_{5,8}\)    & 39                & 9                      & \(d_0\)                  & 2                      & 23             \\
    \(\mu_{12}\I\)           & \(G_{21}\) & \multicolumn{6}{l}{\textit{No data available.}} \\
    \(\mu_{20}\I\)           & \(G_{17}\) & \multicolumn{6}{l}{\textit{No data available.}} \\
    \(\mu_{30}\I\)           & \(G_{18}\) & \multicolumn{6}{l}{\textit{No data available.}} \\
    \(\mu_{60}\I\)           & \(G_{19}\) & \multicolumn{6}{l}{\textit{No data available.}} \\
    \hline
    \(\OT_2\)                & \(G_{12}\) & \(\phi_{2,1}\)    & 3                 & 3                      & \(5d_0\)                 & 3                      & 7              \\
    \(\OT_4\)                & \(G_8\)    & \(\phi_{4,5}\)    & 15                & 7                      & \(5d_0\)                 & 3                      & 11             \\
    \(\OT_6\)                & \(G_{14}\) & \(\phi_{2,4}\)    & 14                & 5                      & \(3d_0\)                 & 2                      & 9              \\
    \(\OT_{12}\)             & \(G_{10}\) & \(\phi_{3,10}'\)  & 36                & 11                     & \(d_0\)                  & 2                      & 14             \\
  \end{tabular}
  \caption{\label{tab:champres}Results of the computations with \textsc{Champ}}
\end{table}
%

\begin{table}[htbp]
  \centering
  \begin{tabular}{c | l || c | l}
    \multirow{2}{*}{\(G_0\)} & Groups containing \(G_0\) as & \multirow{2}{*}{\(G_0\)} & Groups containing \(G_0\) as \\
    & largest reflection group & & largest reflection group\\
    \hline
    \hline
    \(\mu_6\T\) & \(\mu_6 \T\) &
    \(\mu_{12}\T\) & \(\mu_{12} \T\) \\
    \hline
    \(\mu_4\O\) & \(\mu_4 \O\) &
    \(\mu_8\O\) & \(\mu_d \O\), \(d \in \{8, 16\}\)\\
    \(\mu_{12}\O\) & \(\mu_{12} \O\) &
    \(\mu_{24}\O\) & \(\mu_d \O\), \(d \in \{24, 48, 72, 96\}\)\\
    \hline
    \(\mu_4\I\) & \(\mu_4 \I\) &
    \(\mu_6\I\) & \(\mu_6 \I\) \\
    \(\mu_{10}\I\) & \(\mu_{10} \I\) &
    \(\mu_{12}\I\) & \(\mu_d \I\), \(d \in \{12, 24, 36, 48, 72, 84,\)\\
    \(\mu_{20}\I\) & \(\mu_d \I\), \(d \in \{20, 40, 80, 100, 140\}\) & & \phantom{\(\mu_d\I\), \(d\in\{\)}\(96, 108, 132, 144\}\)\\
    \(\mu_{30}\I\) & \(\mu_d \I\), \(d \in \{30, 90, 150\}\) &
    \(\mu_{60}\I\) & \(\mu_d \I\), \(d \in \{60, 120, 180, 240\}\)\\
    \hline
    \(\OT_2\) & \(\OT_2\) &
    \(\OT_4\) & \(\OT_4\) \\
    \(\OT_6\) & \(\OT_6\) &
    \(\OT_{12}\) & \(\OT_{12}\) \\
  \end{tabular}
  \caption{\label{tab:open2}Groups for which there is no answer yet}
\end{table}

\clearpage
\printbibliography
\end{document}